\documentclass[pdflatex,sn-mathphys-ay]{sn-jnl}

\usepackage[utf8]{inputenc}
\usepackage[T1]{fontenc}
\usepackage{graphicx}%
\usepackage{multirow}%
\usepackage{amsmath,amssymb,amsfonts,bm}%
\usepackage{amsthm}%
\usepackage{mathrsfs}%
\usepackage[title]{appendix}%
\usepackage[svgnames]{xcolor}%
\usepackage{textcomp}%
\usepackage{manyfoot}%
\usepackage{booktabs}%
\usepackage{algorithm}%
\usepackage{algorithmicx}%
\usepackage{algpseudocode}%
\usepackage{listings}%
\usepackage{enumerate}

\usepackage{hyperref}
\hypersetup{
  colorlinks,
  citecolor=NavyBlue,
  urlcolor=NavyBlue,
  linkcolor=NavyBlue}

\theoremstyle{theorem}
\newtheorem{theorem}{Theorem}
\newtheorem{proposition}{Proposition}%
\newtheorem{lemma}{Lemma}%
\newtheorem{openproblem}{Open Problem}

\theoremstyle{definition}
\newtheorem{definition}{Definition}%
\newtheorem{assumption}{Assumption}%
\theoremstyle{remark}
\newtheorem{remark}{Remark}%

\usepackage[nameinlink]{cleveref}
\Crefformat{equation}{Eq.~(#2#1#3)}
\crefmultiformat{equation}
  {Eqs.~(#2#1#3)}          
  { and~(#2#1#3)}         
  {, (#2#1#3)}            
  { and~(#2#1#3)}         

\crefname{theorem}{theorem}{theorems}
\Crefname{theorem}{Theorem}{Theorems}

\crefname{proposition}{proposition}{propositions}
\Crefname{proposition}{Proposition}{Propositions}

\crefname{openproblem}{open problem}{open problems}
\Crefname{openproblem}{Open Problem}{Open Problems}

\crefname{assumption}{assumption}{assumptions}
\Crefname{assumption}{Assumption}{Assumptions}

\usepackage{dsfont}
\newcommand \one    {\mathds{1}}
\let        \un     \one

\newcommand \Dset  {\mathbb{D}}
\newcommand \Hset  {\mathbb{H}}
\newcommand \Mset  {\mathbb{M}}
\newcommand \Nset  {\mathbb{N}}
\newcommand \Rset  {\mathbb{R}}
\newcommand \Rplus {\Rset_+}
\newcommand \Rbar  {\overline{\Rset}}
\newcommand \Rpbar {\overline{\Rset}_+}

\newcommand \Pset  {\mathbb{P}}

\newcommand \Tset  {\mathbb{T}}
\newcommand \Xset  {\mathbb{X}}
\newcommand \Yset  {\mathbb{Y}}

\newcommand \Acal  {\mathcal{A}}  
\newcommand \Bcal  {\mathcal{B}}  
\newcommand \Fcal  {\mathcal{F}}  
\newcommand \Pcal  {\mathcal{P}}  
\newcommand \Tcal  {\mathcal{T}}  
\newcommand \Ycal  {\mathcal{Y}}  
\newcommand \Xcal  {\mathcal{X}}  

\newcommand \ddiff    {\mathrm{d}}
\newcommand \dx       {\ddiff x}
\newcommand \dy       {\ddiff y}
\newcommand \ds       {\ddiff s}
\newcommand \dt       {\ddiff t}

\newcommand \dnu      {\ddiff \nu}
\newcommand \dnub     {\ddiff \nub}
\newcommand \dnug     {\ddiff \nug}
\newcommand \dlambda  {\ddiff \lambda}
\newcommand \domega   {\ddiff \omega}

\newcommand \nug      {\boldsymbol{\nu}}
\newcommand \nub      {\overline{\nug}}

\newcommand \Ymsp {\left( \Yset, \Ycal \right)}

\newcommand{\prt}[1]{\left(#1\right)} 

\DeclareMathOperator{\vect}{span}  
\DeclareMathOperator{\dom}{dom}    
\DeclareMathOperator{\var}{var}    
\DeclareMathOperator{\ri}{ri}      

\newcommand \Esp  {\mathsf{E}}
\newcommand \Prob {\mathsf{P}}

\begin{document}

\title[Article Title]{%
  Uncertainty functionals revisited:\\%
  Concavity and Jensen's inequality}

\author*[1]{\fnm{Julien} \sur{Bect}}\email{julien.bect@centralesupelec.fr}

\author[1]{\fnm{Xujia} \sur{Zhu}}\email{xujia.zhu@centralesupelec.fr}

\affil*[1]{\orgdiv{Laboratoire des Signaux et Systèmes},
  \orgname{Université Paris-Saclay, CNRS, CentraleSupélec},
  \orgaddress{\street{3, rue Joliot Curie}, \city{Gif-sur-Yvette},
    \postcode{91190}, \country{France}}}

\abstract{%
  This article presents a theoretical study of uncertainty functionals
  on general measurable spaces. %
  These functionals are fundamental in experimental design and global
  sensitivity analysis, where they are used to quantify variability
  and information content in probabilistic models. %
  As first articulated in DeGroot's seminal 1962 article, a natural
  requirement is that uncertainty should decrease on average when
  additional information is obtained. %
  This requirement is equivalent to the probabilistic
  form of Jensen's inequality on the space of probability measures. %
  Our main results show that concavity is necessary but not sufficient
  for Jensen's inequality to hold whenever the underlying measurable
  space is infinite. %
  We also provide practicable sufficient conditions under which the
  desired property holds. %
  These results contribute to a clearer mathematical foundation for
  uncertainty quantification. %
  Several open questions are formulated.%
}


\maketitle

\section{Introduction}\label{sec1}

The story of uncertainty functionals originates in the late 1950s in
the field of Bayesian experimental design, when \cite{lindley:1956:information,
  lindley:1957:binomial} proposed measuring the information provided by an
experiment about a parameter~$\bm{\theta} \in \Theta$ through the
difference
\begin{equation*}
  \Delta \;=\;
  \Esp\left( I\left( \Prob^{\bm{\theta} \mid X} \right) \right)
  - I\left( \Prob^{\bm{\theta}} \right),
\end{equation*}
where $X$ denotes the result of the experiment, and the
``information'' $I(\nu)$ carried by a probability measure~$\nu$
on~$\Theta$ is defined as the negative of the (differential) entropy
with respect to a given reference measure~$\lambda$:
\begin{equation*}
  I(\nu) \;=\; \int \frac{\dnu}{\dlambda}\, \log\left( \frac{\dnu}{\dlambda} \right)\; \dlambda,
\end{equation*}
assuming that $\nu \ll \lambda$. %
Lindley observes, among other properties, that this difference is always
non-negative---indeed, it coincides with the expected Kullback--Leibler
divergence from the posterior to the prior distribution, or
equivalently to the mutual information between~$\theta$ and~$X$. %
Assuming that $\Theta$ is finite, and $\lambda$ is the uniform
distribution on~$\Theta$, Lindley's information~$I(\nu)$ is the
negative of Shannon's entropy~$H(\nu)$, which is a non-negative measure
of uncertainty, and $\Delta$ can be read as the expected reduction of
uncertainty provided by the experiment:
\begin{equation}\label{equ:Delta-H}
  \Delta \;=\;
  H\left( \Prob^{\bm{\theta}} \right)
   - \Esp\left( H\left( \Prob^{\bm{\theta} \mid X} \right) \right).
\end{equation}
(As is well known, $H(\nu)$ can be negative when $H$ is the
differential entropy in general; nevertheless, even in this case, the
difference~$\Delta$ remains non-negative.)

\cite{degroot:1962:uncertainty} took Lindley's idea one step further,
replacing Shannon's entropy in~\eqref{equ:Delta-H} by a general
uncertainty functional, which we will continue to denote by~$H$ in
this article. %
An uncertainty functional in the sense of DeGroot\footnote{%
  \cite{degroot:1962:uncertainty} actually uses the term ``uncertainty
  function'', which is consistent with the fact that the article
  focuses on the case where~$\Theta$ is finite. %
  We will prefer the term ``uncertainty functional'' in order to
  emphasize that, in general, the space of probability measures
  on~$\Theta$ is infinite-dimensional.} %
is a (measurable) non-negative mapping defined on the set of
probability measures on~$\Theta$, which vanishes on Dirac measures and
yields a non-negative uncertainty reduction~$\Delta$ for any possible
experiment. %
Functionals that satisfy the latter property will be called
\emph{decreasing on average} as in \cite{bect:2019:supermartingale}. %
For probability measures on finite sets, apart from Shannon's entropy,
the $\alpha$-entropies of \cite{havrda:1967:alpha-entropy} and the
Gini-Simpson index \citep{simpson:1949:measurement} are examples of
uncertainty functionals. %
For probability measures on~$\Rset$, variance is a prototypical
example of such a function, which can be seen as a special case of a
more general class of uncertainty functionals called quadratic
entropies by \cite{rao:1984:convexity}. %
The reader is referred to \cite{bernardo:1979:expected,
  rao:1982:diversity, degroot:1986:concepts,
  degroot:1994:changes, ginebra:2007:information, dawid:2007:geometry,
  hainy:2014:learningfunctions} for related
discussions and examples.

Uncertainty functionals also play a role in global sensitivity
analysis, where the objective is to assess how uncertainties in a
model's inputs propagate through the model and contribute to the
variability of its output \citep[see, e.g.,][and references
therein]{Saltelli2004, DaVeiga2021}. %
Consider indeed the popular first-order Sobol' sensitivity indices
associated with a function~$f: \Rset^n \to \Rset$ and a random input
vector~$X$:
\begin{equation*}
  S_j \;=\; \frac{\var\left( \Esp\left( f(X) \mid X_j \right) \right)}{\var\left( f(X) \right)},
  \quad 1 \le j \le n,\vspace{2pt}
\end{equation*}
originally introduced by \citet{sobol:1993:sensitivity} using the
functional ANOVA decomposition of~$f$ \citep{hoeffding:1948}. %
As observed by \cite{homma-saltelli:1996:importance-measures}, these
indices admit a natural interpretation as the (normalized) expected
reduction of uncertainty~\eqref{equ:Delta-H} in the output $Y = f(X)$
that can be achieved by reducing the uncertainty on a single
input~$X_j$: %
\begin{equation}
  \label{equ:Sobol-Sj-reformul}
  S_j \;=\; \frac{\var\left( f(X) \right) - \Esp\left( \var\left( f(X) \mid X_j \right) \right)}{\var\left( f(X) \right)}
  \;=\; \frac{H\left( \Prob^{Y} \right) - \Esp\left( H\left( \Prob^{Y \mid X_j} \right) \right)}{H\left( \Prob^{Y} \right)},
\end{equation}
where $H(\nu)$ denotes the variance of a probability measure~$\nu$. %
Several authors have taken this reformulation as the starting point
for a generalization of Sobol's variance-based sensitivity analysis,
replacing the variance by a broader class of functionals~$H$
\citep{Fort2016,Borgonovo2021,Fissler2023,Straub2025}. %
The special case where $Y = f(X)$ is discrete, and $H$ is Shannon's
entropy, has been proposed by \cite{Krzykacz2001} and further
discussed by \cite{Auder2009}. %
As above, to ensure that the right-hand side
of~\eqref{equ:Sobol-Sj-reformul} defines a sensitivity index
in~$\left[ 0,\, 1 \right]$, interpretable as a fraction of uncertainty
reduction, $H$~must be non-negative and decreasing on average.

Focusing on the case where~$\Theta$ is a finite set,
\cite{degroot:1962:uncertainty} proves that $H$~is decreasing on
average if, and only if, it is concave. %
The ``only if'' part is proved by considering a suitable Bernoulli
experiment, and that part of DeGroot's argument remains valid even if
$\Theta$~is not finite. %
The ``if'' part follows from the integral---i.e., probabilistic---form
of Jensen's inequality, using the fact that the set of probability
measures on a finite set of cardinal~$n$ can be identified with the
$(n-1)$-dimensional probability simplex, which is a convex subset
of~$\Rset^n$. %
That part does not generalize directly beyond the case of finite
sets. %
Indeed, the set of probability measures becomes infinite-dimensional,
and it has been proved by \cite{perlman74} that Jensen's inequality
does not hold in general in infinite-dimensional spaces. %
The main contribution of this article is to prove that concavity is
not sufficient in general for $H$~to be decreasing on average.

The paper is organized as follows. %
\Cref{sec2} introduces the key concepts underlying uncertainty
functionals, provides their formal definitions, and illustrates these
definitions with two well-known families of functionals. %
In particular, we show that the property of being decreasing on
average is equivalent to (the probabilistic form of) Jensen's
inequality on the space of probability measures. %
\Cref{sec3} examines the relationship between concavity and Jensen's
inequality. %
Our main results, based on explicit counterexamples, show that
concavity is necessary but not sufficient for Jensen's inequality to
hold whenever the underlying measurable space is infinite. %
We also discuss a sufficient condition under which Jensen's inequality
holds. %
Finally, \Cref{sec4} summarizes the main findings of this work and
discusses directions for future research.

\section{Uncertainty functionals}\label{sec2}

\subsection{Definitions}

Let $\Ymsp$ denote a measurable space, %
and let $\Pset = \Mset_1^+\Ymsp$ denote the space of all probability
measures on~$\Ymsp$. %
The set~$\Yset$ corresponds to the output space in the sensitivity
analysis setting ($\Yset = \Rset$ in our introductory discussion) and
to the parameter space~$\Theta$ in the experimental design setting. %
We are interested in functionals defined on~$\Pset$, possibly taking
infinite values, i.e., mappings from~$\Pset$
to~$\Rbar = \left[ -\infty, +\infty \right]$. %
We call such functionals \emph{probability functionals} on~$\Ymsp$.

Let $\Fcal$ the $\sigma$-algebra on~$\Pset$ generated by the
evaluation maps~$\pi_B:\nu \mapsto \nu(B)$, $B \in \Ycal$. %
A random element~$\nug$ in~$\left( \Pset, \Fcal \right)$ is called a
random measure \citep[see, e.g.,][p.105--106]{kall02}, and we denote
by~$\nub$ its intensity measure, defined
by~$\nub(B) = \Esp\left( \nug(B) \right)$. %
Let $\left( \Omega, \Acal, \Prob \right)$ denote a probability
space. %
A function $\nug: \Omega \to \Pset$ is a random element
in~$\left( \Pset, \Fcal \right)$ if, and only if,
$(\omega, B) \mapsto \nug(\omega, B) \triangleq \left( \nug(\omega)
\right)(B)$ is a probability kernel from~$\left( \Omega, \Acal \right)$
to~$\left( \Yset, \Ycal \right)$. %
With this notation, the intensity measure~$\nub$ is given
by~$\nub(B) = \int_\Omega \Prob(\domega)\, \nug(\omega, B)$. %
In the rest of the article, a probability
functional~$H: \Pset \to \Rbar$ is called measurable if it is $\Fcal$
/ $\Bcal\left(\Rbar\right)$-measurable. %

\begin{definition}\label{def:satisfy-Jensen}
  We shall say that a measurable probability functional
  $H: \Pset \to \Rbar$ \emph{satisfies Jensen's inequality} (concave
  version) if, for all random probability measures~$\nug$ on~$\Ymsp$ such
  that $H(\nug)$ is quasi-integrable,
  \begin{equation}
    \label{equ:Jensen-ineq-measures}
    H(\nub) \ge \Esp\left( H( \nug) \right).
  \end{equation}
  (Recall that a random variable~$Z$ is called quasi-integrable if
  $\Esp\left( Z^+ \right) \wedge \Esp\left( Z^- \right) < +\infty$.)
\end{definition}

\begin{remark}
  The random variable~$H(\nug)$ is always quasi-integrable if $H$~is
  non-negative, which will be assumed for our uncertainty
  functionals %
  (see \Cref{def:UF} below).%
\end{remark}

Leveraging the correspondence \citep[see, e.g.,][Lemma~1.37]{kall02}
between random measures and probability kernels, we can also %
characterize probability functionals that satisfy Jensen's inequality
in terms of regular conditional distributions. %
Indeed, given a measurable space $\left( \Xset, \Xcal \right)$ and a
random element $Z = (X, Y)$ in
$\left( \Xset \times \Yset, \Xcal \otimes \Ycal \right)$, the regular
conditional distribution~$\Prob^{Y \mid X }$, when it exists, is a
random measure on~$\Ymsp$ defined on the probability
space~$\bigl( \Xset, \Xcal, \Prob^X \bigr)$. %
Conversely, given a random measure $\nug$ on~$\Ymsp$, defined on some
probability space $\left( \Xset, \Xcal, \pi \right)$, we can set
$\Omega = \Xset \times \Yset$, $\Acal = \Xcal \otimes \Ycal$,
$\Prob(\dx\,\dy) = \pi(\dx)\, \nug(x, \dy)$, and then view
$\nug$ as the regular conditional distribution of~$Y:(x,y) \mapsto y$
given $X:(x,y) \mapsto x$. %
The following result then stems immediately from this correspondence.%

\begin{proposition}\label{prop:equiv-Jensen-DoA}
  A measurable probability functional~$H$ satisfies Jensen's
  inequality if, and only if, it is \emph{decreasing on average}: for all
  measurable spaces~$\left( \Xset, \Xcal \right)$, for all random
  elements $Z = (X, Y)$ in
  $\left( \Xset \times \Yset, \Xcal \otimes \Ycal \right)$ such that a
  regular conditional distribution~$\Prob^{Y \mid X }$ exists and
  $H\left( \Prob^{Y \mid X } \right)$ is quasi-integrable,
  \begin{equation}\label{equ:DoA-def}
    H\left( \Prob^Y \right)
    \;\ge\;
    \Esp\left( H\left( \Prob^{Y \mid X } \right) \right).
  \end{equation}
\end{proposition}

\begin{remark}
  A regular conditional distribution always exists if $\Yset$ is
  isomorphic to a Borel subset of~$\left[ 0, 1 \right]$, which
  includes in particular Polish spaces \citep[see, e.g.,][Chapters~1
  and~5]{kall02}.
\end{remark}

\begin{definition}\label{def:UF}
  We will say that a non-negative probability functional
  $H: \Pset \to \Rpbar$ is an \emph{uncertainty functional} if i) it
  is measurable, ii) satisfies Jensen's inequality (equivalently,
  decreases on average), and iii) vanishes on Dirac measures.
\end{definition}

This definition calls for a series of comments.

\begin{remark}
  Following \cite{degroot:1962:uncertainty}, we require in
  \Cref{def:UF} that $H(\delta_y) = 0$ for all $y \in \Yset$, but we
  do not require that $H(\nu) = 0$ if and only if $\nu$~is a Dirac
  measure. %
  In other words, an uncertainty functional measures the uncertainty
  about certain features of an unknown element of~$\Yset$, not
  necessarily about the element itself. %
  For instance, for any measurable space~$\Ymsp$, if
  $\xi: \Yset \to \Rset$ is a measurable function, then
  $H:\nu \mapsto \var_{Y \sim \nu}(\xi(Y))$ is an uncertainty
  functional (see \Cref{sec:regular-QUFs} for further examples in this
  vein). %
  Similarly, the null functional $H: \nu \mapsto 0$ is an uncertainty
  functional---albeit a rather uninformative one. %
  In contrast, in the related definition of a measure of diversity,
  \cite{rao:1982:diversity, rao:1984:convexity} requires that
  $H(\nu) = 0$ if and only if $\nu$ is a Dirac measure.
\end{remark}

\begin{remark}\label{rem:PH-not-convex}
  The set $\Pset_H = \{ \nu \in \Pset \mid H(\nu) < +\infty \}$ is not
  necessarily convex for an uncertainty functional in the sense of
  \Cref{def:UF}. %
  As an example where $\Pset_H$ is not convex, consider the case where
  $\Yset$~is finite or countable, and $H$~is defined by $H(\nu) = 0$
  if $\nu$ is a Dirac measure, and $H(\nu) = +\infty$ otherwise. %
  It is easily checked that $H$~is an uncertainty functional, and
  $\Pset_H$ is the set of all Dirac measures on~$\Yset$, which is not
  convex.
\end{remark}

\begin{remark}
  It is important to allow the value~$+\infty$ when dealing with
  uncertainty functionals on general measurable spaces~$\Ymsp$, if one
  wants to work with functionals defined on the
  set~$\Pset = \Mset_1^+\Ymsp$ of all probability measures
  on~$\Ymsp$. %
  \cite{degroot:1962:uncertainty} does not explicitly
  allow~$+\infty$, but this is not a severe restriction since he
  focuses on the case where~$\Yset$ is a finite set. %
  For probability measures on $\left( \Rset, \Bcal(\Rset) \right)$,
  however, the situation is different: %
  for instance, the variance functional is finite only for probability
  measures with a finite second-order moment. %
  (Alternatively, we could work with real-valued functionals defined
  on a convex subset of~$\Mset_1^+\Ymsp$, which would be slightly more
  restrictive since $\{ H < +\infty \}$ is not necessarily convex, as
  explained in \Cref{rem:PH-not-convex}.)
\end{remark}

\subsection{Some examples}\label{sec:UFexamples}

\subsubsection{\texorpdfstring{$\phi$}{phi}-entropies}
\label{sec:phi-entropies}

As a first example, consider the case where $\Yset$~is a finite or
countable set, endowed with the discrete $\sigma$-algebra
$\Ycal = \Pcal(\Yset)$. %
For a probability measure~$\nu$
on~$\left( \Yset, \Pcal(\Yset) \right)$, set
$\nu_y = \nu\left( \{ y \} \right)$ for $y \in \Yset$. %
The following proposition provides an interesting class of uncertainty functionals
in this setting.

\begin{proposition}
  Let $h: \left[ 0, 1 \right] \to \Rplus$ denote a non-negative
  concave function, such that $h(0) = h(1) = 0$. %
  Then
  \begin{equation}\label{equ:def-phi-entropy}
    H:\, \nu \;\mapsto\; \sum_{y \in \Yset} h\left( \nu_y \right)
  \end{equation}
  is an uncertainty functional on~$\left( \Yset, \Pcal(\Yset) \right)$
  and, for all random probability measures~$\nug$
  on~$\left( \Yset, \Pcal(\Yset) \right)$ such that $\Esp\left( H\left( \nug \right) \right) < +\infty$, we have the relation
  \begin{equation}
    \label{equ:Jensen-difference-phi-entropy}
    H\left( \nub \right) - \Esp\left( H\left( \nug \right) \right)
    \;=\;
    \sum_{y \in \Yset} \Bigl(
      h\left( \nub_y \right) - \Esp\left( h(\nug_y) \right)
    \Bigr), \vspace{-2mm}
  \end{equation}
  where the sum in the right-hand side is well-defined
  in~$\Rpbar$.
\end{proposition}

\begin{proof}
  Obviously, the probability functional~\eqref{equ:def-phi-entropy} is
  measurable and vanishes on Dirac measures. %
  It remains to show that it satisfies Jensen's inequality. %
  We have:
  \begin{equation*}
    H\left( \nub \right)
    \;=\; \sum_{y \in \Yset} h\left( \nub_y \right),
    \qquad
    \Esp\left( H\left( \nug \right) \right)
    \;=\; \Esp\left( \sum_{y \in \Yset} h(\nug_y) \right)
    \;=\; \sum_{y \in \Yset} \Esp\left( h(\nug_y) \right).
  \end{equation*}
  For each $y \in \Nset$, since $h$ is concave, Jensen's inequality
  on~$\Rset^n$ (see \Cref{thm:Jensen-Rp} below) gives $h(\nub_y) \geq \Esp\left( h(\nug_y) \right)$. %
  Thus, we can write
  $h(\nub_y) = \left( h(\nub_y) - \Esp\left( h(\nug_y) \right) \right)
  + \Esp\left( h(\nug_y) \right)$, where the two terms in the
  right-hand side are non-negative. %
  Summing over $\Yset$ yields:
  \begin{equation*}
    H\left( \nub \right)
    \;=\; \sum_{y \in \Yset} \Bigl(
      h\left( \nub_y \right) - \Esp\left( h(\nug_y) \right)
    \Bigr) \,+\, \Esp\left( H\left( \nug \right) \right).
  \end{equation*}
  Jensen's inequality and \Cref{equ:Jensen-difference-phi-entropy} follow.
\end{proof}

Uncertainty functionals of this form have been considered for a long
time in the literature (usually on finite sets, but the extension to
countable sets is straightforward), under various names such as
$\phi$-entropies \citep{Burbea1982, bental:1986:rate-disto} or
trace-form generalized entropies \citep{scarfone:2013:entropic,
  tempesta:2015:trace-form}. %
Important special cases are obtained with $h(p) = -p\, \log(p)$ and
$h(p) = (\alpha - 1)^{-1} \left( p - p^\alpha \right)$, $\alpha > 1$,
corresponding respectively to the classical Shannon (or
Boltzmann-Gibbs) entropy,
$H(\nu) = - \sum_y \nu_y\, \log\left( \nu_y \right)$, and to the
$\alpha$-entropy of \cite{havrda:1967:alpha-entropy} and
\cite{tsallis:1988:generalization},
$H(\nu) = (\alpha - 1)^{-1} \bigl( 1 - \sum_y \nu_y^\alpha \bigr)$. %
The special case $\alpha = 2$ corresponds to the Gini-Simpson
diversity index \citep{simpson:1949:measurement},
$H(\nu) = 1 - \sum_y \nu_y^2$.

\subsubsection{Regular quadratic uncertainty functionals}
\label{sec:regular-QUFs}

For any symmetric measurable function
$\rho: \Yset \times \Yset \to \Rplus$, the expression
\begin{equation}
  \label{eq:regular-QUF-def}
  H(\nu) \;=\; \iint \rho(y, y')\, \nu(\dy)\, \nu(\dy')
\end{equation}
defines a measurable probability functional $H: \Pset \to \Rpbar$. %
Uncertainty functionals of this form have been studied, for example,
by \cite{rao:1984:convexity}, under the name \emph{quadratic entropy},
and by \cite{hainy:2014:learningfunctions}, who calls
\emph{distance-based information function} the negative of such an
uncertainty functional. %
More generally, functionals of the
form~\eqref{eq:regular-QUF-def}---not necessarily uncertainty
functionals---are called \emph{regular statistical functionals of
  degree~$2$} in the literature on U-statistics \citep[see,
e.g.,][]{lee:2019:U-stat}.

\begin{remark}
  In the terminology of potential theory, $H(\nu)$ corresponds to the
  negative of the \emph{energy} of~$\nu$ \cite[see,
  e.g.,][]{pronzato2021minimum} associated with the kernel~$\rho$. %
  Note, however, that classical potential theory typically deals with
  singular kernels that are infinite on the diagonal, such as the
  Newtonian kernel $\rho(x, x') \propto \lVert x - x' \rVert^{-1}$
  on~$\Rset^3$. %
  In contrast, we shall be mainly interested in kernels that vanish on
  the diagonal (see \Cref{thm:regular-QUF}).
\end{remark}

\begin{definition}
  Any uncertainty functional of the form~\eqref{eq:regular-QUF-def} is
  called a \emph{regular quadratic uncertainty functional} (regular
  QUF). %
  The symmetric measurable function
  $\rho: \Yset \times \Yset \to \Rpbar$ is called the \emph{kernel} of
  the regular QUF.
\end{definition}

Let $\Pset_H = \left\{ \nu \in \Pset \mid H(\nu) < +\infty \right\}$,
and $\Mset_H = \vect\left( \Pset_H \right)$ denote the vector
space of signed measures spanned by~$\Pset_H$. %
The following result provides a characterization of regular QUFs in
terms of properties of their kernel.

\begin{theorem}\label{thm:regular-QUF}
  Let $H: \Pset \to \Rpbar$ denote a measurable, non-negative
  functional of the form~\eqref{eq:regular-QUF-def}, %
  where $\rho: \Yset \times \Yset \to \Rplus$ is symmetric and
  measurable. %
  Then, $H$~is a regular QUF if, and only if, \vspace{-2mm}
  \begin{enumerate}[i) ]
    
  \item $\rho$ vanishes on the diagonal (i.e., $\rho(y, y) = 0$ for
    all $y \in \Yset$), and

  \item $\rho$ is \emph{conditionally integrally negative definite} (CIND), i.e., %
    \begin{equation}
      \label{equ:CIND}
      \forall \nu \in \Mset_H,\quad
      \nu(\Yset) = 0 \;\Rightarrow\; \iint \rho(y, y')\, \nu(\dy)\, \nu(\dy') \le 0.
    \end{equation}

  \end{enumerate}
  Moreover, if $H$ is a regular QUF, then $\Pset_H$ is convex, 
  and for all random probability
  measures~$\nug$ on~$\left( \Yset, \Pcal(\Yset) \right)$ such
  that~$H(\nub) < +\infty$, we have
  \begin{equation}
    \label{equ:Jensen-difference-regular-QUF}
    H\left( \nub \right) - \Esp\left( H\left( \nug \right) \right)
    \;=\;
    -\, \Esp\left[
      \iint \rho(y, y')\,
      \left( \nug - \nub \right)(\dy)\,
      \left( \nug - \nub \right)(\dy')
    \right].
  \end{equation}
\end{theorem}

\begin{proof}
  See \Cref{sec:proof-app:regular-QUF} for a proof. %
  Note: \citet[][Lemma~3.2]{pronzato2020bayesQuad} proved a similar result under the stronger assumption that the kernel is
  bounded. %
  The proof presented in \Cref{sec:proof-app:regular-QUF} has a lot in
  common with that of \citeauthor{pronzato2020bayesQuad}, but it further
  generalizes the argument to cover unbounded kernels as well.
\end{proof}

Several important uncertainty functionals can be recovered in this
framework, including the variance functional
$H(\nu) = \int y^2 \nu(\dy) - \left( \int y\, \nu(\dy) \right)^2$ on
$\Yset = \Rset$ with $\rho(y, y') = \frac{1}{2} (y - y')^2$, %
Gini's mean difference \citep{yitzhaki2003gini} with
$\rho(y, y') = | y - y' |$, %
or more generally the fractional Brownian variance of
\cite{szekely:2013:energy} with
$\rho(y, y') = \lVert y - y' \rVert^\beta$, $0 < \beta < 2$, on
$\Yset = \Rset^p$ . %
The Gini-Simpson index $H(\nu) = 1 - \sum_y \nu_y^2$ on a finite or
countable set~$\Yset$, already discussed in \Cref{sec:phi-entropies},
is also recovered here with $\rho(y, y') = \one_{y \neq y'}$.

Additionally, given a $\sigma$-finite measure space
$(\Tset, \Tcal, \mu)$ and a measurable function
$\xi: \Yset \times \Tset \to \Rset$, the kernel
\begin{equation*}
  \rho(y, y') \;=\; \frac{1}{2} \int_\Tset \left( \xi(y, t) - \xi(y', t) \right)^2\, \mu(\dt)
\end{equation*}
yields the integrated variance functional:
\begin{equation}\label{equ:int-var}
  H(\nu) \;=\; \int_\Tset \var_\nu\left( \xi(Y, t) \right)\, \mu(\dt),
\end{equation}
where the index~$\nu$ indicates that~$Y \sim \nu$. %
The Gini-Simpson index and the integrated Bernoulli variance
functional \citep{bect:2019:supermartingale} are both special cases of
\Cref{equ:int-var}, where $\xi$ takes its values in~$\{ 0, 1 \}$.

\bigbreak

The term CIND, introduced in \Cref{thm:regular-QUF}, does not appear
to have been used previously in the literature. %
It is, however, consistent with the %
terminology used by \cite{pronzato2023blue}, since $\rho$ is CIND in
the sense of \Cref{thm:regular-QUF} if and only if it is
\emph{conditionally integrally positive definite} (CIPD) in the sense
of \cite{pronzato2023blue}. %
See also \cite{pronzato2020bayesQuad, pronzato2021minimum} and
\cite{sriperumbudur2011universality} for related terminology.%

\begin{proposition}\label{prop:CIND-equiv-CND}
  Let $\rho: \Yset \times \Yset \to \Rplus$ be symmetric, measurable
  and vanishing on the diagonal. %
  Then $\rho$~is CIND if, and only if, it is \emph{conditionally
    negative definite} (CND), that is, for all $m \ge 1$, all points $y_1$,
  \ldots, $y_m \in \Yset$, and all $\lambda_1$, \ldots,
  $\lambda_m \in \Rset$ such that $\sum_{k=1}^m \lambda_k = 0$,
  \begin{equation}
    \sum_{k,l}^{m} \lambda_k \lambda_l\, \rho(y_k, y_l) \;\le\; 0.%
  \end{equation}
\end{proposition}

\begin{remark}
  It follows from classical results on CND functions that $\rho$ is
  symmetric, CND and vanishes on the diagonal if, and only,
  there exists a symmetric positive definite
  kernel~$k: \Yset \times \Yset \to \Rset$ such that
  $\rho(y, y') = \lVert k(y,\cdot) - k(y',\cdot) \rVert^2_\Hset$,
  where $\Hset$ is the reproducing kernel Hilbert space with
  kernel~$k$ %
  (see \Cref{thm:CND-kernel} and \Cref{rem:RKHS-feat-map} for
  details). %
  In this setting, for all~$\nu \in \Pset_H$, the kernel embedding
  $k_\nu = \int k(y,\cdot)\, \nu(\dy)$ is well-defined (see
  \Cref{lem:regular-QUF-embedding} and
  \Cref{rem:kernel-mean-embedding}). %
  Moreover, for all random probability measures~$\nug$
  on~$\left( \Yset, \Pcal(\Yset) \right)$ such that
  $\nub \in \Pset_H$, we have $\nug \in \Pset_H$ almost surely, and it
  is easily checked that the right-hand side of
  \Cref{equ:Jensen-difference-regular-QUF} is equal to
  $\Esp\left( \lVert k_{\nug} - k_{\nub} \rVert_\Hset^2 \right)$. %
  In other words, the (unnormalized) sensitivity indices based on a
  regular QUF coincide with the MMD-based indices of
  \cite{daveiga:2021}.
\end{remark}

\section{Concavity and Jensen's inequality}\label{sec3}

\subsection{Main results and an open problem}

We turn now to the problem of characterizing the probability
functionals that satisfy Jensen's inequality. %
More specifically, we investigate the relation between Jensen's
inequality and concavity. %
This is indeed a natural question, considering the well-known
equivalence, for functions defined on~$\Rset^n$, between concavity and
Jensen's inequality. %
We recall here this equivalence, in form that is suitable for
comparison with \Cref{def:satisfy-Jensen}.

\begin{theorem}[Jensen's inequality on~$\Rset^n$]\label{thm:Jensen-Rp}
  Let $f:\Rset^n \to \Rbar$ be a measurable function. %
  Then, $f$ is concave if and only if, for all random vectors~$X$ in
  $\Rset^n$ such that $\Esp\left( \lVert X \rVert \right) < +\infty$
  and $f(X)$ is quasi-integrable, %
  the inequality $f\left( \Esp(X) \right) \ge \Esp\left( f(X) \right)$
  holds.
\end{theorem}

\begin{remark}
  For concave functions with values in~$\Rset \cup \{ -\infty \}$, the
  statement can be strengthened: in this case,
  $\Esp(f^+(X)) < +\infty$, and thus $f(X)$ is automatically
  quasi-integrable.
\end{remark}

Concavity for probability functionals is defined as follows. (This is
the usual definition of concavity for extended-real valued functions,
when $\Pset$ is seen as a convex subset of the set of bounded signed
measures on~$\Ymsp$.)

\begin{definition}\label{def:concavity}
  A probability functional $H:\Pset \to \Rbar$ is called
  \emph{concave} if, for all $\lambda \in (0,1)$ and all $\nu_1$,
  $\nu_2 \in \dom(H) = \left\{ \nu \in \Pset \mid H(\nu) > -\infty
  \right\}$,
  \begin{equation}
    H\prt{\lambda \nu_1 + (1-\lambda) \nu_2}
    \;\geq\;
    \lambda H(\nu_1) + (1-\lambda) H(\nu_2).
  \end{equation}
\end{definition}

In the case where $\Yset$ is a finite set, \citet[Theorem
2.1]{degroot:1962:uncertainty} established, for probability
functionals, the equivalence between concavity and the property of
being decreasing on average (which is equivalent, by
\Cref{prop:equiv-Jensen-DoA}, to Jensen's inequality). %
In fact, a part of \citeauthor{degroot:1962:uncertainty}'s argument
remains valid for a general measurable space~$\Ymsp$, which yields the
following statement.

\begin{theorem}[\citealp{degroot:1962:uncertainty}]\label{thm:degroot62}
  Let $H$ denote a probability functional on a measurable
  space~$\Ymsp$. %
  If $H$ satisfies Jensen's inequality, then $H$ is concave. %
  Conversely, if $H$ is concave and $\Yset$ is finite, then $H$
  satisfies Jensen's inequality.
\end{theorem}

\begin{proof}
  We follow \citeauthor{degroot:1962:uncertainty}'s original proof,
  with minor adaptations.

  Let $H$ be a probability functional that satisfies Jensen's
  inequality. %
  Let $\lambda \in \left( 0, 1 \right)$ and %
  $\nu$, $\nu' \in \dom(H)$. %
  Let $\nug$ be a random measure equal to~$\nu$ with
  probability~$\lambda$, and $\nu'$ otherwise. %
  Then $\nub = \lambda\,\nu + (1-\lambda)\,\nu'$, and $H(\nug)$ is
  quasi-integrable (since $\Esp(H^-(\nug)) < +\infty$). %
  Applying Jensen's inequality to~$\nug$ yields the concavity of~$H$.

  Conversely, assume that $\Yset$ has finite cardinality~$n$, and let
  $H$~be a concave probability functional on~$\Ymsp$ for some
  $\sigma$-algebra $\Ycal$ on~$\Yset$. %
  Then the set~$\Pset$ of all probability measures on~$\Ymsp$ can be
  identified measurably with a convex subset~$C$ of the
  $(n-1)$-dimensional probability simplex in~$\Rset^n$. %
  Thus, $H$ can be identified with a concave function on~$\Rset^n$
  (using the value~$-\infty$ to extend outside~$C$), %
  and therefore Jensen's inequality for the probability functional~$H$
  follows from \Cref{thm:Jensen-Rp}. %
  (Note: $\Pset$ coincides exactly with the probability simplex
  if~$\Ycal = \Pcal(\Yset)$.)
\end{proof}

\citet[Theorem 2]{hainy:2014:learningfunctions} states a
generalization of \Cref{thm:degroot62}, formulated using the property
of being decreasing on average (recall \Cref{prop:equiv-Jensen-DoA})
rather than Jensen's inequality (\Cref{def:satisfy-Jensen}), which
removes the finiteness restriction on~$\Yset$ in the converse. %
The proposed proof, however, is not sufficient: %
in particular, it does not specify how the use of \Cref{thm:Jensen-Rp}
can be replaced, since in general $\Pset$ can no longer be identified
with a convex subset of some Euclidean space, when $\Yset$~is not
finite.

We are now in a position to state our main results, which prove that
the equivalence between concavity and Jensen's inequality does not
hold for probability functionals on general measurable spaces. %
Our first result is based on an earlier counterexample, proposed by
\cite{perlman74} in the setting of convex functionals defined on
topological vector spaces. %
Our second result relies on a different type of counterexample, which
is new to the best of our knowledge.

\begin{theorem}\label{thm:main-result-1}
  Assume that
  $\Ymsp \cong \left( \Nset, \Pcal(\Nset) \right) \otimes \left(
    \Yset', \Ycal' \right)$, for some measurable
  space~$\left( \Yset', \Ycal' \right)$. %
  Then, there exists a concave, measurable probability functional
  on~$\Ymsp$ that does not satisfy Jensen's inequality.
\end{theorem}

\begin{theorem}\label{thm:main-result-2}
  Assume that
  $\Ymsp \cong \left( \Rset, \Bcal(\Rset) \right) \otimes \left(
    \Yset', \Ycal' \right)$, for some measurable
  space~$\left( \Yset', \Ycal' \right)$. %
  Then, there exists a concave, non-negative, measurable probability
  functional on~$\Ymsp$ that vanishes on Dirac measures but does not
  satisfy Jensen's inequality.
\end{theorem}

\begin{proof}
  For both results we can assume without loss of generality
  that~$\left( \Yset', \Ycal' \right)$ is trivial, i.e., that
  $\Ymsp \cong \left( \Nset, \Pcal(\Nset) \right)$ for
  \Cref{thm:main-result-1} and
  $\Ymsp \cong \left( \Rset, \Bcal(\Rset) \right)$ for
  \Cref{thm:main-result-2}. %
  A concave probability functional that does not satisfy Jensen's
  inequality is constructed, for each of the two settings, in
  \Cref{sec:counterexample1,sec:counterexample2} respectively.
\end{proof}

\Cref{thm:main-result-2} proves that the requirement that $H$~should
satisfy Jensen's inequality, in the definition of an uncertainty
functional, cannot be replaced by a simple concavity requirement as
soon as the underlying measurable space $\Ymsp$ ``carries'' a
real-valued random variable (in the sense that
$\Ymsp \cong \left( \Rset, \Bcal(\Rset) \right) \otimes \left( \Yset',
  \Ycal' \right)$, for some measurable
space~$\left( \Yset', \Ycal' \right)$). %
In the case of uncertainty functionals
on~$\left( \Nset, \Pcal(\Nset) \right)$, the question remains open.

\begin{openproblem}
  Is it possible to construct a concave, non-negative, measurable
  probability functional on $\left( \Nset, \Pcal(\Nset) \right)$ that
  vanishes on Dirac measures but does not satisfy Jensen's inequality?
\end{openproblem}

\subsection{Counterexamples for Jensen's inequality}\label{sec:counterexample}

\subsubsection{%
  A first counterexample (proof of %
  \texorpdfstring{%
    \Cref{thm:main-result-1}%
  }{Theorem~\ref{thm:main-result-1}})}\label{sec:counterexample1}

In this section, we present a counterexample adapted from
\cite{perlman74} as a proof of \Cref{thm:main-result-1}. %
Take $\Yset = \Nset$, and let $H$ denote the real-valued probability
functional defined by
\begin{equation}\label{eq:CE1}
  H(\nu) =
  \begin{cases}
    g(\nu)  &\text{ if } g(\nu) >-\infty, \\
    \;\;0 &\text{ if } g(\nu)=-\infty,
  \end{cases}
  \quad \text{where } g(\nu) = \liminf_{y \to+\infty} \left( - y^2 \nu_y \right).
\end{equation}
Note that $H$~takes values in~$\left( -\infty,\, 0 \right]$.

We begin by establishing concavity. For any $\nu,\nu' \in \Pset$ and
$\lambda\in(0,1)$, if $g\prt{\lambda\nu+(1-\lambda)\nu'}=-\infty$,
we have $H(\lambda\nu + (1-\lambda)\nu') = 0$ which reaches the
maximum of $H$, and thus clearly satisfies the desired inequality. In
the case where the $g\prt{\lambda\nu+(1-\lambda)\nu'}$ is finite, we have
\begin{equation*}
  \begin{split}
    -\infty \;<\; H(\lambda\nu + (1-\lambda)\nu')
    & \;=\; \liminf_{y\to+\infty} 
      \left[ -y^2 \prt{\lambda\nu_y + (1-\lambda) \nu'_y} \right]\\
    & \;\leq\; \min\left\{%
      \lambda \liminf_{y\to+\infty} ( -y^2 \nu_y),\;
      (1-\lambda) \liminf_{y\to+\infty} (-y^2 \nu'_y)
      \right\},
  \end{split}
\end{equation*}
which implies that $g(\nu)$ and $g(\nu')$ are both finite,
indicating that we have both $H(\nu)=g(\nu)$ and
$H(\nu')=g(\nu')$. %
Hence, the superadditivity and positive homogeneity of the $\liminf$
operator imply that
\begin{equation*}
  \begin{split}
    H(\lambda\nu + (1-\lambda)\nu')
    & \;=\; g(\lambda\nu + (1-\lambda)\nu') \\
    & \;\geq\; \lambda g(\nu) + (1-\lambda)g(\nu') = \lambda H(\nu) + (1-\lambda)H(\nu').
  \end{split}        
\end{equation*}
	
To see that Jensen's inequality does not hold for this functional,
consider a discrete random variable~$X$ taking values in~$\Nset$,
with distribution
\begin{equation*}
  \Prob\prt{X=x} = \frac{6}{\pi^2} \frac{1}{(1+x)^2},\quad x \in \Nset,
\end{equation*}
and the random measure $\nug = \delta_{X }$,
where $\delta_x$ denotes the Dirac measure at~$x$. %
In this setting we have
\begin{equation*}
  H\prt{\nub}
  = H\prt{\Prob^X}
  = \frac{6}{\pi^2}\liminf_{y \to +\infty} \left( -\frac{y^2}{(1 + y)^2}\right)
  = -\frac{6}{\pi^2},
\end{equation*}
and $H\prt{\nug} = 0$ since $H$~vanishes on Dirac measures. %
Therefore %
$H\prt{\nub} < \Esp\left( H\prt{\nug} \right)$, which contradicts
Jensen's inequality.

\begin{remark}
  In fact, the function $g$ alone is sufficient to show \Cref{thm:main-result-1}.
  Its range covers the extended non-positive real line, whereas the
  range of $H$ is essentially the same, except that it excludes
  $-\infty$. This construction is intended to highlight that the 
  issue does not stem from the infinite value. %
\end{remark}

\subsubsection{%
  Another counterexample (proof of %
  \texorpdfstring{\Cref{thm:main-result-2}}{Theorem~7})}\label{sec:counterexample2}

In this section, we construct a concave probability functional that
satisfies all the requirements of an uncertainty functional except
Jensen's inequality, thereby proving \Cref{thm:main-result-2}.

Let $\Yset = \Rset$ and $\Ycal = \Bcal(\Rset)$. %
Consider the integrated variance functional defined by
\begin{equation}\label{eq:CE2}
  H(\nu) \;=\; \int_\Rset \var_\nu\left( \one_{Y=s} \right)\, \mu(\ds),
\end{equation}
where $\mu$ denotes the counting measure on~$(\Rset, \Bcal(\Rset))$. %
The probability functional $H$~is measurable, nonnegative, and
vanishes on Dirac measures. %
It is also concave, since for each~$s \in \Rset$ the mapping
$\nu \mapsto \var_\nu \left( \one_{Y=s} \right)$ is concave (and in
fact satisfies Jensen's inequality).

To see that Jensen's inequality does not hold for~$H$, observe that
\begin{equation}\label{eq:CE2b}
  H(\nu) \;=\; \int_\Rset \nu_s\, \left( 1 - \nu_s \right)\, \mu(\ds)
  \;=\; \sum_{s \in \Rset} \nu_s\, \left( 1 - \nu_s \right),
\end{equation}
where $\nu_s = \nu\left( \{ s \} \right)$ as before. %
Consider the random probability measure
$\nug = \frac{1}{2}\left( \delta_{-X} + \delta_X \right)$, where $X$
follows an exponential distribution. %
Then $\nub$ is a Laplace distribution, which has no atom and hence
$H(\nub) = 0$. %
On the other hand, $H(\nug) = \frac{1}{2}$ almost surely. %
Therefore Jensen's inequality does not hold:
\begin{equation*}
  \Esp\left(H\prt{\nug} \right) = \frac{1}{2} > 0 = H(\nub).  
\end{equation*}

\begin{remark}
  The essence of this counterexample is the lack of
  $\sigma$-finiteness of the counting measure~$\mu$ on~$\Rset$. %
  Indeed, as mentioned in \Cref{sec:regular-QUFs}, any probability
  functional of the form~\eqref{eq:CE2} with $\mu$ a $\sigma$-finite
  measure is an uncertainty functional, which is an easy consequence
  of Fubini--Tonelli's theorem:
  \begin{align*}
    H(\nub)
    & \;=\; \int_\Rset \var_{\nub}\left( \one_{Y=s} \right)\, \mu(\ds)
      \;\ge\; \int_\Rset \Esp\left( \var_{\nug}\left( \one_{Y=s} \right)\right)\, \mu(\ds) \\
    & \;=\; \Esp\left( \int_\Rset \var_{\nug}\left( \one_{Y=s} \right)\, \mu(\ds) \right)
      \;=\; \Esp\left( H(\nug) \right).
  \end{align*}
\end{remark}

\begin{remark}
  It is easy to see from \Cref{eq:CE2b} that $H$~is upper-bounded by
  one. %
  More precisely, $\sup H = 1$ and the sup is not attained.
\end{remark}

\subsection{A sufficient condition, and more open problems}
\label{sec:regular-QUFs2}

We focus in this section on non-negative probability functionals. %
A sufficient condition for Jensen's inequality to hold in this setting
is provided by the following result.

\begin{proposition}[see, e.g., %
  \citealp{bect:2019:supermartingale}, Proposition 3.17]%
  \label{prop:expected-loss}
  Let $\Dset$ be a non-empty set. %
  Let $L: \Yset \times \Dset \to \Rpbar$ be a function such that, for
  all $d \in \Dset$, $L(\cdot, d)$ is measurable. %
  Then, the expression
  \begin{equation}
    \label{equ:expr-H-expected-loss}
    H(\nu) = \inf_{d \in \Dset} \int L(y,d)\, \nu(\dy),
    \quad \nu \in \Pset,
  \end{equation}
  defines a non-negative measurable probability functional
  $H:\Pset \to \Rpbar$ that satisfies Jensen's inequality.
\end{proposition}

\begin{remark}
  The probability function~\eqref{equ:expr-H-expected-loss} is an
  uncertainty functional in the sense of \Cref{def:UF} if, and only
  if, $\inf L(y, \cdot) = 0$ for all~$y \in \Yset$.
\end{remark}

\begin{remark}
  Sensitivity indices based on probability functionals of the
  form~\eqref{equ:expr-H-expected-loss} have been proposed by
  \cite{Fort2016}, \cite{Borgonovo2021}, \cite{Fissler2023}
  and~\cite{Straub2025}.
\end{remark}

Interpreting $d \in \Dset$ as a decision, $L$ as a loss function, and
$\nu$ as a posterior distribution in a Bayesian setting, the integral
in the right-hand-side of \Cref{equ:expr-H-expected-loss} corresponds
to the posterior risk, and the uncertainty~$H(\nu)$ to the minimal
posterior risk. %
This idea of measuring uncertainty using an expected loss---or
information using an expected utility---has a long history, which goes
back to the early days of statistical decision theory \citep[see,
e.g.,][in particular Eq.~(2.14) and the associated
discussion]{mccarthy:1956:measures, degroot:1962:uncertainty}. %
See also \cite{bernardo:1979:expected},
\cite{degroot:1986:concepts,degroot:1994:changes},
\cite{dawid:2007:geometry} and \cite{ginebra:2007:information}. %
The following question, however, remains to the best of our knowledge
unanswered (see below for a partial answer).

\begin{openproblem}\label{op:inf}
  Does there exist a measurable space $\Ymsp$, and a non-negative
  measurable probability functional~$H$ on~$\Ymsp$, such that
  $H$~satisfies Jensen's inequality but cannot be written in the
  form~\eqref{equ:expr-H-expected-loss}?
\end{openproblem}

As a special case, consider now the situation where the infimum in
\Cref{equ:expr-H-expected-loss} is attained for all~$\nu \in \Pset$. %
It turns out that this corresponds to a well-known class of
uncertainty functionals \citep[see, e.g.,][]{grunwald:2004,
  dawid:2007:geometry, gneiting:2007:strictly}:

\begin{proposition}
  A probability functional~$H$ admits the
  representation~\eqref{equ:expr-H-expected-loss}, where the infimum on~$\Dset$ is
  attained for all~$\nu \in \Pset$ if, and only if, it is the
  \emph{generalized entropy function} associated with a proper, %
  non-negative, negatively-oriented scoring rule
  $S: \Pset \times \Yset \to \Rpbar$.%
\end{proposition}

\begin{proof}
  We reproduce the argument from \citet[Section~3.4]{grunwald:2004}.

  Assume that $H$ has the form~\eqref{equ:expr-H-expected-loss}, and
  the infimum is attained for all~$\nu \in \Pset$. %
  Then, there is a function $\nu \mapsto d_\nu^*$ from~$\Pset$
  to~$\Dset$ such that, for all $\nu \in \Pset$,
  \begin{equation}\label{equ:H-min-scoring}
    H(\nu) \;=\; \int L(y, d_\nu^*)\, \nu(\dy)
    \;=\; \min_{\nu' \in \Pset} \int L(y, d_{\nu'}^*)\, \nu(\dy).
  \end{equation}
  Setting $S(y, \nu) = L(y, d^*_\nu)$ yields the desired
  negatively-oriented scoring rule, which takes values
  in~$\Rpbar$, and is proper since the minimum in
  \Cref{equ:H-min-scoring} is attained at $\nu' = \nu$.

  Conversely, assume that there exists a proper, non-negative, negatively-oriented
  scoring rule $S: \Pset \times \Yset \to \Rpbar$
  such that, for all $\nu \in \Pset$,
  \begin{equation*}
    H(\nu)
    \;=\; \int S(y, \nu)\, \nu(\dy)
    \;=\; \min_{\nu' \in \Pset} \int S(y, \nu')\, \nu(\dy).
  \end{equation*}
  Then the representation~\eqref{equ:expr-H-expected-loss} holds with
  $\Dset = \Pset$ and $L = S$.
\end{proof}

\begin{openproblem}\label{op:min}
  Does there exist a measurable space $\Ymsp$, and a non-negative
  measurable probability functional~$H$ on~$\Ymsp$, such that
  $H$~can be written in the
  form~\eqref{equ:expr-H-expected-loss}, but does not have an 
  associated proper scoring rule?
\end{openproblem}

We conclude this section with a result that provides a partial answer
to \Cref{op:inf,op:min}: the answer to both questions is negative in
the case of a finite set~$\Yset$.

\begin{theorem}\label{thm:min-finite-case}
  Let $\Yset$ be a finite set, endowed with a
  $\sigma$-algebra~$\Ycal$. %
  Let $H$ denote a non-negative, concave probability
  functional on~$\Ymsp$. %
  Then, $H$~can be written in the
  form~\eqref{equ:expr-H-expected-loss}, in such a way that the
  infimum is attained for all~$\nu \in \Pset$. %
\end{theorem}

\begin{proof}
  See \Cref{sec:proof-app:min-finite-case}.
\end{proof}

\begin{remark}
  It is critical, for this result to hold, to allow the
  value~$+\infty$ for the function~$L$
  in~\eqref{equ:expr-H-expected-loss}. %
  A special case of this result is proved by
  \cite{degroot:1962:uncertainty}, under the additional assumption
  that $H$, seen as a function on the probability simplex, is
  continuous. %
  In this case, $L$ can be chosen to take only finite values, and the
  result is a straightforward consequence of a classical result in
  convex analysis---namely, that a continuous concave function on a
  convex subset of~$\Rset^n$ admits at each point an exact affine
  majorant.
\end{remark}

\section{Conclusion}\label{sec4}

In this article, we have provided a unified theoretical treatment of
uncertainty functionals on general measurable spaces, going beyond
DeGroot's seminal work on finite spaces. %
We have proved that concavity is not, in general, a sufficient
condition for Jensen's inequality to hold in its probabilistic
form---a key property in applications such as Bayesian experimental
design and global sensitivity analysis. %
In particular, our second counterexample demonstrates that Jensen's
inequality can fail to hold even for non-negative bounded concave
functionals that vanish on Dirac measures. %

Several important questions remain open. %
In particular, to further clarify the theory of uncertainty
functionals, it is important to understand whether all uncertainty
functionals arise as the generalized entropy function of some proper
scoring rule. %
Another important direction for future research is to refine the
analysis of uncertainty functionals in the context of global
sensitivity analysis. %
Indeed, the property of being decreasing on average is necessary but
not sufficient to yield a complete set of sensitivity indices à la
Sobol', in which interactions are separated from main effects in a
meaningful way.

\backmatter

\begin{appendices}

\section{Proofs}\label{sec:appendix:proofs}

\subsection{Proof of \texorpdfstring{\Cref{thm:regular-QUF}}{Theorem~\ref{thm:regular-QUF}}}
\label{sec:proof-app:regular-QUF}


\begin{theorem}[see, e.g., \citealp{berg84harmonic}, Chapter~3]
  \label{thm:CND-kernel}
  Let $\rho:\Yset \times \Yset \to \Rset$ be symmetric and vanishing
  on the diagonal. %
  Then, $\rho$ is CND if, and only if, there exists a Hilbert
  space~$\Hset$ and a function 
  $\psi: \Yset \to \Hset$, such that
  $\rho(y, y') = \lVert \psi(y) - \psi(y') \rVert^2$ for all $y$,
  $y' \in \Yset$.
\end{theorem}

We will consider in Appendices~\ref{sec:proof-app:regular-QUF}
and~\ref{sec:proof-app:CIND-equiv-CND} a measurable kernel~$\rho$, a
Hilbert space~$\Hset$ and a function~$\psi: \Yset \to \Hset$ with the
same properties as in \Cref{thm:CND-kernel}, which we summarize for
reference:%

\begin{assumption}\label{assumpt:rho-CND}
  $\rho:\Yset \times \Yset \to \Rset$ is a symmetric and measurable
  CND kernel, vanishing on the diagonal, %
  and $\psi: \Yset \to \Hset$ is such that
  $\rho(y, y') = \lVert \psi(y) - \psi(y') \rVert^2$ for all $y$,
  $y' \in \Yset$.%
\end{assumption}

\begin{remark}\label{rem:RKHS-feat-map}
  Given a kernel~$\rho$ that satisfies \Cref{assumpt:rho-CND}, a
  corresponding function $\psi: \Yset \to \Hset$ for~$\rho$ is obtained, for
  instance, by considering a reproducing kernel~$\Hset$ Hilbert space
  on~$\Yset$, with symmetric positive definite kernel~$k$ such that
  $\rho(y, y') = k(y, y) + k(y', y') - 2 k(y, y')$, and setting
  $\psi(y) = k(y, \cdot)$. %
  It is well known that such a kernel~$k$ always exists if~$\rho$ is
  symmetric, CND and vanishes on the diagonal: %
  for instance, take
  $k(y, y') = \frac{1}{2} \left( \rho(y, y_0) + \rho(y', y_0) - \rho(y, y') \right)$ for some
  arbitrary $y_0 \in \Yset$. %
  (Note that $k$~is measurable.)%
\end{remark}

\begin{lemma}\label{lem:regular-QUF:weak-meas}
  Let \Cref{assumpt:rho-CND} hold. %
  Then $\psi$ is weakly measurable, i.e.,
  $y \mapsto \left< h,\, \psi(y) \right>$ is measurable for
  all~$h \in \Hset$.%
\end{lemma}

\begin{proof}
  Assume first that $\Hset$ and $\psi$ are as in
  \Cref{rem:RKHS-feat-map}. %
  Then, for all $h \in \Hset$ of the form $h = k(y_0, \cdot)$, we have
  $\left< h,\, \psi(y) \right> = k(y_0, y)$, and therefore
  $y \mapsto \left< h,\, \psi(y) \right>$ is measurable since $k$~is
  measurable. %
  The result extends to linear combinations, and then to
  all~$h \in \Hset$ by density. %
  Therefore, $\psi$~is weakly measurable.

  Assume now that $\tilde \psi: \Yset \to \tilde\Hset$ provides
  another representation of the same kernel~$\rho$. %
  Then there exists an isometric isomorphism between the closure
  of~$\vect\left\{ \psi(y) - \psi(y_0),\, y \in \Yset \right\}$
  in~$\Hset$ and the closure
  of~$\vect\left\{ \tilde\psi(y) - \tilde\psi(y_0),\, y \in \Yset
  \right\}$ in~$\tilde\Hset$, from which it is easy to see that
  $\tilde\psi$ is weakly measurable.
\end{proof}

\begin{lemma}\label{lem:regular-QUF:set-PH}
  Assume that \Cref{assumpt:rho-CND} holds, and let $\nu \in \Pset$.
  Then $\iint \rho\, \ddiff(\nu\otimes\nu) < +\infty$ if, and only if,
  $\int \lVert \psi \rVert^2\, \dnu < +\infty$. %
  As a consequence, the set
  $\Pset_H = \left\{ \nu \in \Pset \mid \iint \rho\,
    \ddiff(\nu\otimes\nu) < +\infty \right\}$ is convex.
\end{lemma}  

\begin{proof}
  Preliminary remark: since %
  $\psi$ is weakly measurable by \Cref{lem:regular-QUF:weak-meas}, %
  $y \mapsto \lVert \psi(y) \rVert^2$ is measurable %
  as well, since
  $\lVert \psi(y) \rVert^2 = \sup_{\lVert h \rVert \le 1} \left< h,\,
    \psi(y) \right>$. %

  Let $\nu \in \Pset$ be such %
  that $\int \lVert \psi \rVert^2\, \dnu < +\infty$. %
  Then we have
  $\rho(y, y') \le 2\, \bigl( \lVert \psi(y) \rVert^2 + \lVert
  \psi(y') \rVert^2 \bigr)$ and thus
  \begin{equation*}
    \iint \rho\, \ddiff(\nu\otimes\nu)
    \;\leq\; 4\, \int \lVert \psi \rVert^2\, \dnu \;<\; +\infty.
  \end{equation*}
  Conversely, assume that
  $\iint \rho\, \ddiff(\nu\otimes\nu) < +\infty$. %
  Then, for $\nu$-almost all~$y_0 \in \Yset$, we have
  $\int \rho(y, y_0)\, \nu(\dy) < +\infty$. %
  For any $y_0 \in \Yset$ such that this inequality holds, we get
  \begin{align*}
    \int \lVert \psi \rVert^2\, \dnu
    & \;\le\; 2\, \int \left[ \lVert \psi(y) - \psi(y_0) \rVert^2 +  \lVert \psi(y_0) \rVert^2 \right]\, \nu(\dy)\\
    & \;=\; 2\, \int \rho(y, y_0)\, \nu(\dy) + 2\, \lVert \psi(y_0) \rVert^2 \;<\; +\infty.
      \tag*{\qed}
  \end{align*}
  \let\qed\relax
\end{proof}

\begin{lemma}\label{lem:regular-QUF:BH-on-PH-first}
  Let \Cref{assumpt:rho-CND} hold. Then, %
  for all $\nu_1, \nu_2 \in \Pset_H$,
  $\rho \in L^1(\nu_1 \otimes \nu_2)$.
\end{lemma}  

\begin{proof}
  Let $\nu_1, \nu_2 \in \Pset_H$. Set
  \begin{equation*}
    B(\nu_1, \nu_2)
    \;=\; \iint_{\Yset \times \Yset} \rho(y,y')\, \nu_1(\dy)\, \nu_2(\dy')
    \;\in\; \left[ 0, +\infty \right],
  \end{equation*}
  and note that $H(\nu) = B(\nu, \nu)$. %
  The set $\Pset_H$ is convex by \Cref{lem:regular-QUF:set-PH}, and
  therefore $H\left( \frac{1}{2}(\nu_1 + \nu_2) \right) < +\infty$ for
  all $\nu_1, \nu_2 \in \Pset_H$. %
  Besides, $B$ is symmetric since~$\rho$ is symmetric. %
  Thus
  \begin{align*}
    H\left( \frac{\nu_1 + \nu_2}{2} \right)
    & \;=\; \iint \rho\, \ddiff\left[
      \frac{1}{4}\, \nu_1 \otimes \nu_1 + \frac{1}{4}\, \nu_1 \otimes \nu_2
      + \frac{1}{4}\, \nu_2 \otimes \nu_1 + \frac{1}{4}\, \nu_2 \otimes \nu_2
      \right]\\
    & \;=\; \frac{1}{4}\, H(\nu_1) + \frac{1}{4}\, H(\nu_2) + \frac{1}{2}\, B(\nu_1, \nu_2),
  \end{align*}
  which proves that $B(\nu_1, \nu_2)$ is finite.
\end{proof}

\bigbreak

\begin{proof}[Proof of \Cref{thm:regular-QUF}]
  To begin with, observe that $H(\delta_y) = \rho(y, y)$ for
  all~$y \in \Yset$. %
  Thus, it is clear that $H$~vanishes on Dirac measures if, and only
  if, its kernel~$\rho$ vanishes on the diagonal. %
  We assume from now on that $\rho: \Yset \times \Yset \to \Rpbar$ is
  symmetric, measurable, and vanishes on the diagonal.

  Assume that $H$~is an uncertainty functional, and therefore
  satisfies Jensen's inequality. %
  Then, in particular, $H$~is concave by \Cref{thm:degroot62}, %
  which implies that $\rho$ is conditionally negative definite
  (CND). %
  To see it, set $\nu = \sum_{k=1}^m \lambda_k \delta_{y_k}$. %
  Assume without loss of generality that~$\nu \neq 0$. %
  Then there exists finitely supported probability measures $\nu_1$,
  $\nu_2 \in \Pset_H$, and $\alpha > 0$, such that
  $\nu = \alpha (\nu_1 - \nu_2)$. %
  Setting $\bar\nu = \frac{1}{2} (\nu_1 + \nu_2)$, it is easy to check
  that %
  \begin{equation*}
    \sum_{k,l} \lambda_k \lambda_l \rho(y_k, y_l)
    \;=\; \iint \rho\, \ddiff(\nu \otimes \nu)
    \;=\; 4\alpha^2 \left( \frac{1}{2}(H(\nu_1)+H(\nu_2)) - H(\bar\nu) \right)
    \;\le\; 0,
  \end{equation*}
  which proves that $\rho$ is CND. %
  Now let $\nu \in \Mset_H$ be such that $\nu(\Yset) = 0$. %
  Again, assume without loss of generality that~$\nu \neq 0$, since
  \Cref{equ:CIND} trivially holds for $\nu = 0$. %
  Then there exists $\nu_1$, $\nu_2 \in \Pset_H$, and $\alpha > 0$,
  such that $\nu = \alpha (\nu_1 - \nu_2)$. %
  We know from \Cref{lem:regular-QUF:BH-on-PH-first} that
  $\rho \in L^1(\nu_1 \otimes \nu_2)$, and therefore
  $\rho \in L^1( \left| \nu \otimes \nu \right| )$. %
  Finally, setting $\bar\nu = \frac{1}{2} (\nu_1 + \nu_2)$, we have %
  \begin{equation*}
    \iint \rho\, \ddiff(\nu \otimes \nu)
    \;=\; 4\alpha^2 \left( \frac{1}{2}(H(\nu_1)+H(\nu_2)) - H(\bar\nu) \right)
    \;\le\; 0,
  \end{equation*}
  which proves that $\rho$~is CIND.
  
  Conversely, assume that $\rho$~is CIND, and therefore CND. %
  Let $\nug$ be a random probability measures
  on~$\left( \Yset, \Pcal(\Yset) \right)$. %
  Assume without loss of generality that $H(\nub) < +\infty$, i.e.,
  $\nub \in \Pset_H$ (if not, Jensen's inequality holds trivially). %
  Then, by \Cref{lem:regular-QUF:set-PH}, we have
  $\int \lVert \psi \rVert^2\, \dnub < +\infty$. %
  Since
  $\int \lVert \psi \rVert^2\, \dnub = \Esp\bigl( \int \lVert \psi
  \rVert^2\, \dnug )$, we also have
  $\int \lVert \psi \rVert^2\, \dnug < +\infty$ almost surely, hence
  $\nug \in \Pset_H$ almost surely. %
  The integral
  $\iint \rho\, \ddiff\left[ (\nug - \nub) \otimes (\nug - \nub)
  \right]$ in the right-hand side
  of~\Cref{equ:Jensen-difference-regular-QUF} is therefore
  well-defined almost surely. %
  Using the bilinearity of~$\otimes$ and the symmetry of~$\rho$, we
  have:
  \begin{equation*}
    - \iint \rho\, \ddiff\left[ (\nug - \nub) \otimes (\nug - \nub) \right]
    \;=\; 2 \iint \rho\, \ddiff (\nug \otimes \nub) - H(\nug) - H(\nub).
  \end{equation*}
  Finally, taking expectations we obtain
  \begin{equation}
    - \Esp\left( \iint \rho\, \ddiff\left[ (\nug - \nub) \otimes (\nug - \nub) \right] \right)
    \;=\;  2 H(\nub) - \Esp\left( H(\nug) \right) - H(\nub)
    \;=\; H(\nub) - \Esp\left( H(\nug) \right),
  \end{equation}
  where we have used Fubini--Tonelli's theorem and the fact that
  $\Esp\left( \int \rho( \cdot, y' )\, \dnug \right) = \int \rho(
  \cdot, y' )\, \dnub$. %
  The left-hand side is non-negative since $\rho$ is~CIND,
  $\nug - \nub \in \Mset_H$ almost surely and
  $(\nug - \nub)(\Yset) = 0$, which concludes the proof.
\end{proof}

\subsection{Proof of \texorpdfstring{\Cref{prop:CIND-equiv-CND}}{Proposition~\ref{prop:CIND-equiv-CND}}}
\label{sec:proof-app:CIND-equiv-CND}

\begin{lemma}\label{lem:regular-QUF:set-MH}
  Let \Cref{assumpt:rho-CND} hold. %
  Let $\Mset_H$ be defined as in \Cref{sec:regular-QUFs}, i.e.,
  $\Mset_H = \vect(\Pset_H)$, with
  $\Pset_H = \left\{ \nu \in \Pset \mid \iint \rho\,
    \ddiff(\nu\otimes\nu) < +\infty \right\}$. %
  Then we have
  \begin{equation}\label{equ:MH-two-face}
    \Mset_H %
    \;=\; \bigl\{ \nu \in \Mset \bigm| \textstyle\iint \rho\; \ddiff
    |\nu\otimes\nu| < +\infty \bigr\} %
    \;=\; \bigl\{ \nu
    \in \Mset \bigm| \textstyle\int \lVert \psi \rVert^2\, \ddiff |\nu| < +\infty \bigr\},
  \end{equation}
  where $\Mset$ denotes the space of bounded signed measures on~$\Yset$.
\end{lemma}  

\begin{proof}
  By \Cref{lem:regular-QUF:set-PH} we have the equivalence
  \begin{equation}\label{equ:equiv-lemma1}
    \iint \rho\, \ddiff(\nu\otimes\nu) < +\infty
    \quad\Leftrightarrow\quad
    \int \lVert \psi \rVert^2\, \dnu < +\infty
  \end{equation}
  for all $\nu \in \Pset$, which extends immediately to non-negative
  bounded measures.

  Let $\nu \in \Mset_H = \vect(\Pset_H)$. %
  Then there exists $\nu_1$, $\nu_2 \in \Pset_H$ and $\alpha$,
  $\beta \ge 0$ such that $\nu = \alpha \nu_1 - \beta \nu_2$. By
  minimality of the Hahn-Jordan decomposition, it follows that
  \begin{equation*}
    \int \lVert \psi \rVert^2\, \ddiff |\nu|
    \;=\; \int \lVert \psi \rVert^2\, \ddiff \nu^+ + \int \lVert \psi \rVert^2\, \ddiff \nu^-
    \;\le\; \alpha\, \int \lVert \psi \rVert^2\, \ddiff \nu_1 + \beta \int \lVert \psi \rVert^2\, \ddiff \nu_2
    \;<\; +\infty.
  \end{equation*}
  Conversely, any $\nu \in \Mset$ such that
  $\int \lVert \psi \rVert^2\, \ddiff |\nu| < +\infty$ can be written
  using the Hahn-Jordan decomposition as
  $\nu = \alpha \nu_1 - \beta \nu_2$, where $\nu_1$ and~$\nu_2$ are
  probability measures such that
  $\int \lVert \psi \rVert^2\, \ddiff \nu_j <+\infty$,
  $j \in \{ 1, 2 \}$, %
  which implies that $\nu_1$, $\nu_2 \in \Pset_H$, and therefore
  $\nu \in \vect(\Pset_H)$.

  So far we have proved that
  $\Mset_H = \bigl\{ \nu \in \Mset \bigm| \textstyle\int \lVert \psi
  \rVert^2\, \ddiff |\nu| < +\infty \bigr\}$. %
  The other part of \Cref{equ:MH-two-face} then follows from
  \Cref{equ:equiv-lemma1}, using the fact that
  $\left| \nu \otimes \nu \right| = |\nu| \otimes |\nu|$.
\end{proof}

\begin{lemma}\label{lem:regular-QUF-embedding}
  Let \Cref{assumpt:rho-CND} hold, and let $\nu \in \Mset_H$. %
  Then $\psi$ is Pettis-integrable with respect to~$\nu$, i.e., there
  exists an element of~$\Hset$, denoted by~$\int\psi\, \dnu$, such that
  $\left< h,\, \int\psi\, \dnu \right> = \int \left< h,\, \psi(y)
  \right> \nu(\dy)$ for all~$h \in \Hset$.
\end{lemma}  

\begin{proof}
  For all $h \in \Hset$ and $y \in \Yset$,
  $y \mapsto \left< h,\, \psi(y) \right>$ is measurable by
  \Cref{lem:regular-QUF:weak-meas}, and
  $\left| \left< h,\, \psi(y) \right> \right| \le \lVert h \rVert\,
  \lVert \psi(y) \rVert = \lVert h \rVert\, \sqrt{k(y,y)}$. %
  Therefore, for all $\nu \in \Mset_H$ and all $h \in \Hset$, it holds
  that
  \begin{align*}
    \int \left| \left< h,\, \psi(y) \right> \right| \ddiff |\nu|
    & \;\le\;
      \lVert h \rVert\, \int \sqrt{k(y,y)}\, |\nu|(\dy)\\
    & \;\le\;
      \lVert h \rVert\, \left( |\nu|(\Yset) \right)^{\frac{1}{2}}
      \left( \int k(y,y)\, |\nu|(\dy) \right)^{\frac{1}{2}}
      \;<\; +\infty.
  \end{align*}
  As a consequence, the linear functional $\Hset \to \Rset$,
  $h \mapsto \int \left< h,\, \psi(y) \right>\, \nu(\dy)$ is
  well-defined and continuous, and the existence of the Pettis
  integral~$\int \psi\, \dnu$ follows from the Riesz representation
  theorem.
\end{proof}

\begin{remark}\label{rem:kernel-mean-embedding}
  If $\Hset$ is an RKHS and $\psi(y) = k(y, \cdot)$ as in
  \Cref{rem:RKHS-feat-map}, then for all~$\nu \in \Mset_H$, we have
  $\int k(y,y)\, \nu(\dy) = \int \lVert \psi \rVert^2\, \dnu
  <+\infty$, and the Pettis-integral $\int\psi\, \dnu$ coincides with
  the kernel mean embedding $\int k(\cdot, y)\, \nu(\dy)$, where the
  integral is now defined in the Lebesgue sense.
\end{remark}

\begin{proof}[Proof of \Cref{prop:CIND-equiv-CND}]
  Let $\rho: \Yset \times \Yset \to \Rplus$ be symmetric, measurable
  and vanishing on the diagonal. %
  If $\rho$ is CIND, then it is also CND, since finitely supported
  measures belong to~$\Mset_H$. %
  Conversely, assume that $\rho$ is CND. %
  Let $\Hset$ and~$\psi$ be as in \Cref{thm:CND-kernel}. %
  Then $\int \psi\, \dnu$ exists in the Pettis sense by
  \Cref{lem:regular-QUF-embedding}, for all $\nu \in \Mset_H$. %
  Assume in addition that $\nu(\Yset) = 0$. %
  Then we have
  \begin{align*}
    \iint \rho\, \ddiff (\nu \otimes \nu)
    & \;=\; \iint \lVert \psi(y) - \psi(y') \rVert^2\, \nu(\dy)\, \nu(\dy')\\
    & \;=\; \iint \left[ \lVert \psi(y) \rVert^2 + \lVert \psi(y') \rVert^2 - 2 \left< \psi(y),\, \psi(y') \right> \right] \nu(\dy)\, \nu(\dy')\\
    & \;=\; \int \left[ 0 + \int \lVert \psi(y') \rVert^2 \nu(\dy') - 2 \left< \psi(y),\, \textstyle\int \psi\, \dnu \right> \right] \nu(\dy)\\
    & \;=\; - 2\, \bigl\Vert \textstyle\int \psi\, \dnu \bigr\Vert^2 \;\le\; 0,
  \end{align*}
  which proves that $\rho$~is CIND.
\end{proof}

\subsection{Proof of \texorpdfstring{\Cref{thm:min-finite-case}}{Theorem~\ref{thm:min-finite-case}}}
\label{sec:proof-app:min-finite-case}

For all~$m \in \Nset^*$, we denote by $\Sigma_{m-1}$ the
$(m-1)$-dimensional probability simplex in~$\Rset^m$, %
that is,
$\nu_{m-1}=\{\nu \in \Rset^m\mid \sum_{i=1}^m \nu_i = 1,\, \nu_i\geq
0,\text{ for } i=1,\ldots,m\}$. %
Note that the letter~$\nu$, which denotes elsewhere in the paper a
probability measure, is used in this section to denote an element of
the probability simplex. %
The relative interior of~$\Sigma_{m-1}$, denoted by
$\ri(\Sigma_{m-1})$, consists of all probability vectors with strictly
positive components, namely,
$\ri(\Sigma_{m-1})=\{\nu \in \Rset^m\mid \sum_{i=1}^m \nu_i = 1,\,
\nu_i> 0,\text{ for } i=1,\ldots,m\}$.%

\begin{lemma}\label{lem:concave}
  Let $m \in \Nset^*$. %
  Let $H: \Sigma_{m-1} \to \Rpbar$ denote a concave function on the
  probability simplex, taking non-negative extended-real values. %
  Then there exists a set $\Phi \subset \Rpbar^m$ such that %
  i) for all $\nu \in \Sigma_{m-1}$,
  \begin{equation}\label{equ:lem-concave}
    H(\nu) \;\le\; \inf_{\varphi \in \Phi} \left< \varphi,\, \nu \right>,
  \end{equation}
  and ii) for all $\nu \in \ri(\Sigma_{m-1})$, the infimum is attained
  and the inequality is an equality.
\end{lemma}

\begin{proof}
  Assume first that $H(\nu) = +\infty$ for at least one
  $\nu \in \Sigma_{m-1}$. %
  Then, by concavity, $H(\nu) = +\infty$ for all
  $\nu \in \ri(\Sigma_{m-1})$. %
  As a consequence, the claim holds with
  $\Phi = \{ +\infty\, \un \}$, with $\un$ 
  denoting the constant function over the 
  simplex taking the value $1$.

  Assume now that $H(\nu) < +\infty$ for all $\nu \in \Sigma_{m-1}$. %
  It is always true that
  \begin{equation}\label{equ:lem-concave:proof:1}
    H(\nu) \;\le\; \inf_a a(\nu),
  \end{equation}
  where the infimum runs over all affine majorants of~$H$. %
  Moreover, it is a well-known fact from convex analysis that, since
  $H$~is concave, %
  then on the relative interior~$\ri(\Sigma_{m-1})$
  the function is continuous, the
  inequality~\eqref{equ:lem-concave:proof:1} is an equality, and the
  infimum in the right-hand side is attained. %
  To conclude the proof, observe that for
  $a: \nu \mapsto \left< \alpha,\, \nu \right> + \beta$ we have
  $a(\nu) = \left< \alpha + \beta\, \un,\, \nu \right>$ for
  all~$\nu \in \Sigma_{m-1}$, %
  and the components
  of~$\alpha + \beta\, \un$ are non-negative since
  $\alpha_j + \beta = \left< \alpha + \beta\, \un,\, e_j \right> \ge
  H(e_j) \ge 0$ for all $j \in \{ 1, \ldots, m \}$, where $e_j$ is the
  $j$-th element of the canonical basis of~$\Rset^m$ %
  (corresponding to a Dirac measure on the $j$-th element of~$\Yset$).%
\end{proof}

\begin{proof}[Proof of \Cref{thm:min-finite-case}]
  Let $m \in \Nset^*$ denote the cardinal of~$\Yset$, %
  and let $y_1$, \ldots, $y_m$ denote the elements of~$\Yset$. %
  Assume without loss of generality
  that~$\Ycal = \Pcal(\Yset)$---otherwise, we simply need to work with
  the atoms of~$\Ycal$ instead of~$\Yset$ itself. %
  Identify $\Pset$ with the $(m-1)$-dimensional probability simplex
  in~$\Sigma_{m-1} \subset \Rset^m$, and $H$ with a non-negative,
  concave function on~$\Sigma_{m-1}$. %
  The integral $\int L(y,d)\, \nu(\dy)$ in the right-hand side
  of~\eqref{equ:expr-H-expected-loss} becomes, through this
  identification, the (extended) scalar product
  $\left< \varphi_d,\, \nu \right> = \sum_{k=1}^m \varphi_{d,k}\,
  \nu_k$, where
  $\varphi_d = \left( L(y_k, d) \right)_{1 \le k \le m} \in \Rset^m$
  and the convention $(+\infty) \cdot 0 = 0$ is used when
  $\varphi_{d,k} = +\infty$ and~$\nu_k = 0$. %
  The claim to be proved can thus be reformulated as
  \begin{align*}
    (P_m): \quad
    & \text{%
      for all non-negative, concave functions $H: \Sigma_{m-1} \to \Rpbar$,
      there exists a set}\\[-3pt]
    & \text{%
      $\Phi \subset \Rpbar^m$ such that, for all $\nu \in \Sigma_{m-1}$,
      $H(\nu) = \min_{\varphi \in \Phi} \left< \varphi,\, \nu \right>$.}
  \end{align*}

  We prove the result by induction on~$m \in \Nset^*$. %
  For $m = 1$, $\Sigma_{m-1} = \Sigma_0 = \{ 1 \}$, and therefore
  $P_1$ trivially holds with $\Phi = \{ H(1) \}$. %
  Assume now that $P_m$ holds for some~$m \in \Nset^*$. %
  Let $H: \Sigma_m \to \Rpbar$ be concave. %
  By \Cref{lem:concave}, there exists $\Phi_0 \subset \Rpbar^{m+1}$
  such that
    \begin{equation}
    H(\nu) \;\le\; \inf_{\varphi \in \Phi_0} \left< \varphi,\, \nu \right>,
  \end{equation}
  for all $\nu \in \Sigma_m$, with equality and the infimum attained on~$\ri(\Sigma_m)$. %
  Let $F_1$, \ldots, $F_{m+1}$ denote the $(m-1)$-dimensional faces
  of~$\Sigma_m$, where
  $F_k = \left\{ \nu \in \Sigma_m \mid \nu_k = 0 \right\}$. %
  Then $\Sigma_m \setminus \ri(\Sigma_m) = \cup_{k=1}^{m+1} F_k$, and
  each $F_k$~is linearly isomorphic to~$\Sigma_{m-1}$. %
  More precisely, we have $F_k = A_k \Sigma_{m-1}$, where
  $A_k \in \Rset^{(m+1) \times m}$ is the matrix with entries
  in~$\{ 0, 1 \}$ such that $\nu \mapsto A_k \nu$ inserts a zero in
  the $k$-th position, and $\nu \mapsto A_k^\top \nu$ is the inverse
  mapping. %
  Thus, for each $k \in \{ 1, \ldots, m+1 \}$,
  $H_k^0: \nu \mapsto H(A_k \nu)$ is a concave function
  on~$\Sigma_{m-1}$, %
  and therefore by~$P_m$ there exists a set
  $\Phi_k^0 \subset \Rpbar^m$ such that
  $H_k^0(\nu) = \min_{\varphi \in \Phi_k^0} \left< \varphi,\, \nu
  \right>$, for all $\nu \in \Sigma_{m-1}$. %
  As a consequence, for all $\nu \in F_k$ we have
  \begin{equation}
    H(\nu)
    \;=\; H_k^0\left( A_k^\top \nu \right)
    \;=\; \min_{\varphi \in \Phi_k^0} \bigl< \varphi,\, A_k^\top \nu \bigr>
    \;=\; \min_{\varphi \in \Phi_k^0} \bigl< A_k \varphi,\, \nu \bigr>.
  \end{equation}
  Set
  $\Phi_k = \bigl\{ \varphi \in \Rpbar^{m+1} \bigm| \exists \varphi^0
  \in \Phi_k^0,\; \varphi = A_k \varphi^0 + (+\infty)\, e_k \bigr\}$,
  where $e_k$ is the $k$-th element of the canonical basis. %
  Then, for all $\nu \in \Sigma_m$, we have
  \begin{equation*}
    \min_{\varphi \in \Phi_k} \left< \varphi,\, \nu \right>
    \;=\;
    \begin{cases}
      H(\nu) & \text{if $\nu \in F_k$,}\\
      +\infty & \text{otherwise.}
    \end{cases}
  \end{equation*}
  Setting $\Phi = \Phi_0 \cup \Phi_1 \cup \cdots \cup \Phi_{m+1}$
  proves $P_{m+1}$. %
  The infimum is attained for some~$\varphi \in \Phi_0$ for
  $\nu \in \ri(\Sigma_m)$, and for some~$\varphi \in \Phi_k$ is
  $\nu \in F_k \subset \Sigma_m \setminus(\ri(\Sigma_m))$.
\end{proof}

\end{appendices}

\bibliography{bect-zhu-moda14-paper.bib}

\end{document}